\documentclass[11pt,reqno]{amsart}

\usepackage{graphicx,enumerate, stmaryrd, color, url}
\usepackage{amsmath}
\usepackage{amssymb}
\usepackage{amsthm}
\usepackage{amsfonts}
\usepackage{epsfig,amscd,amsthm,mathtools,fourier}
\usepackage{tikz-cd}
\usepackage{tikz}
\usepackage{fullpage}
\usepackage[all]{xy}
\usepackage{comment}
\usepackage{times}
\usepackage{bm}
\usepackage[hidelinks]{hyperref}
\usepackage{appendix}
\usepackage{float}
\usepackage[margin=1.3in,marginparwidth=1.1in,marginparsep=0.1in]{geometry}
\usepackage{marginnote}
\usepackage{import}
\usepackage[normalem]{ulem}

\newtheorem{thm}{Theorem}[section]

\newtheorem{lma}[thm]{Lemma}
\newtheorem{prop}[thm]{Proposition}

\newtheorem{cor}[thm]{Corollary}
\newtheorem*{ThmA}{Theorem A}
\newtheorem*{ThmB}{Theorem B}

\theoremstyle{definition}
\newtheorem{defn}[thm]{Definition}
\newtheorem{Alg}[thm]{Algorithm}
\newtheorem{rmk}[thm]{Remark}

\newtheorem{ex}[thm]{Example}

\newcommand{\W}{W_{\mathrm{Dhar}}}

\newcommand{\md}{\underline{d}}
\newcommand{\me}{\underline{e}}

\newcommand{\val}{\operatorname{val}}
\newcommand{\vall}{\operatorname{val}_{\mathrm{wl}}}

\newcommand{\Pic}{{\operatorname{{Pic}}}}

\newcommand{\Div}{{\operatorname{{Div}}}}

\newcommand{\dv}{{\operatorname{{div}}}}

\newcommand{\RR}{\mathbb{R}}

\newcommand{\ZZ}{\mathbb{Z}}

\newcommand{\mk}{\underline{k}}

\title{Special divisors in special divisor classes on graphs} 
\author{Karl Christ}

\thanks{The author was partially supported by NSF FRG grant DMS–2053261 and is member of the INdaM group GNSAGA.}

\address[Christ]{Dipartimento di Matematica\\
	Università di Torino\\Via Carlo Alberto 10 \\10123 Turin\\  Italy }\email{karl.christ@unito.it}

\begin{document}

\begin{abstract}
We generalize the notion of divisors that are reduced with respect to a vertex to divisors that are reduced with respect to a set of vertices. We establish properties that remain valid in this more general context which allow us to improve the standard algorithm to determine whether a divisor class is effective.  
We then characterize reduced divisors in our sense in terms of Luo's potential-theoretic generalization of reduced divisors. Finally, as a further application, we use this setup to study the existence of so called uniform representatives in special divisor classes.
\end{abstract}

\maketitle

\section{Introduction}
The study of divisors on algebraic curves is one of the classic topics of algebraic geometry. More recently, a divisor theory for finite graphs has been developed that shares striking similarities with the theory for algebraic curves. Starting with \cite{BN}, these analogies have been studied extensively and partially explained; see, for example, \cite{MZ08, Baker, CV10, AC13, CLM, AP20, FJP, MUW, CPS}.

Concretely, a divisor $\md$ on a vertex-weighted graph $G$ is simply a $\mathbb Z$-linear combination of vertices of $G$. Two divisors on $G$ are called linearly equivalent if one can be obtained from the other by a series of so-called chip-firing moves. Thinking about $\md$ as a collection of chips on each vertex, a chip-firing move based at a vertex $v$ is given by moving one chip of $\md$ along each edge adjacent to $v$.  

For a vertex $v$ of $G$ there exists a unique $v$-reduced divisor in any linear equivalence class. Roughly speaking, it is a divisor with the minimal possible non-negative number of chips away from $v$. These representatives have been introduced already in \cite{BN} and remain one of the main tools in studying the Baker-Norine rank of divisors on graphs. One of the reasons for their usefulness in this context is that a divisor class contains an effective representative, i.e. one with all values non-negative, if and only if the $v$-reduced representative is effective.

In the current paper, we generalize the notion of $v$-reduced divisors by replacing the single vertex $v$ with a subset of vertices $V$. We show that some of the properties that make $v$-reduced divisors so useful generalize to this more general setting, and use them to improve the standard algorithm due to Baker and Shokrieh \cite{BS13} to find effective representatives in chip-firing classes or certify that no such representatives exist. We then give a precise relation of our notion with a different generalization of $v$-reduced divisors due to Luo \cite{L13}. Finally, we use $V$-reduced representatives to study the existence of so called uniform representatives in special divisor classes, that are used in the subsequent paper \cite{BCM} to give algebro-geometric applications. 

\subsection*{Main results}

Our main object of interest in this paper is the following generalization of $v$-reduced divisors (introduced independently by Myrla Barbosa \cite{B21} in her PhD thesis): 

\begin{defn}
    Let $V$ be a set of vertices of a graph $G$. A divisor $\md$ is $V$-reduced if it is effective away from $V$ and a chip-firing move along any set of vertices $W$ with $W \cap V = \emptyset$ turns $\md$ negative on at least one vertex in $W$.  
\end{defn}

Different from $v$-reduced divisors, $V$-reduced divisors need not be unique in their divisor class, and it is also not true that a divisor class contains an effective representative if and only if the $V$-reduced representatives are effective. 
But it is easy to see that every divisor class contains a $V$-reduced representative for any $V$ (Lemma~\ref{lma: existence}). And the following proposition gives a relation with effective representatives, see Proposition~\ref{prop:Dhar decomp effective}:

\begin{prop} \label{prop:main}
Let $\md$ be a divisor on $G$ that is not effective but whose linear equivalence class $\delta$ is effective. Let $V \subset V(G)$ be the set of vertices, on which $\md$ has negative coefficients. Then \[\W(\md, V) \neq \emptyset.\]     
\end{prop}

Here $\W(\md, V)$ denotes the Dhar decomposition of $G$ with respect to $\md$ and $V$ which we introduce in Section~\ref{subsec:Dhar}. Roughly speaking, it is the maximal set of vertices along which a chip-firing move of $\md$ does not create any new negative values for $\md$. 

We use Proposition~\ref{prop:main} to give a new algorithm to produce an effective divisor linearly equivalent to a given divisor $\md$, or to certify that no such effective divisor exists. It is a modification of the commonly used algorithm due to Baker and Shokrieh \cite[Algorithm 4]{BS13}, which skips the step in which a divisor is constructed that is effective away from $v$ and instead directly proceeds by calculating successively $V_i$-reduced representatives for a decreasing chain of sets of vertices $V_i$ obtained by chip-firing along the $\W(\md, V_i)$; see Algorithm~\ref{algorithm} for details. 

Next, we compare the notion of $V$-reduced divisors to a different generalization of $v$-reduced divisors introduced by Luo \cite{L13} that generalizes the potential-theoretic characterization of $v$-reduced divisors given in \cite{BS13}, see Theorem~\ref{thm: equivalence reducedness}:

\begin{ThmA}
    A divisor $\md$ of degree $d$ is $V$-reduced if and only if there exists an $\mathbb R$-divisor $\mathcal E$ of degree $d$ supported on $V$, such that $\md$ is $\mathcal E$-reduced in the sense of Luo.
\end{ThmA}

Here an $\mathbb R$-divisor is simply a divisor with values in $\mathbb R$ instead of $\mathbb Z$. An $\mathcal E$-reduced divisor on the other hand is a divisor which, roughly speaking, has minimal distance to $\mathcal E$ in its linear equivalence class; see Section~\ref{sec:potential} for details.  

Finally, as an application, we study the existence of so called uniform divisors in linear equivalence classes of effective divisors whose residual is also effective. To this end, we need to introduce some notation. Recall that the canonical divisor $\mk \in \omega$ is given by \[\mk_v = 2g_v - 2 + \val(v),\] where $g_v$ is the weight of the vertex $v$ and $\val(v)$ its valence. We denote the linear equivalence class of $\mk$ by $\omega$.

Given a divisor $\md$ in the equivalence class $\delta$, its \emph{residual divisor} is defined as $\md^* \coloneqq \mk - \md$. Similarly, its \emph{residual class} is $\delta^* \coloneqq \omega - \delta$.
A divisor class $\delta$ is called \emph{special} if both $\delta$ and its residual class $\delta^* \coloneqq \omega - \delta$ are effective. That is, both $\delta$ and $\delta^*$ need to contain a divisor that has non-negative value on each vertex.
If both $\md$ and its residual $\md^*$ are effective, we call $\md$ a \emph{uniform divisor}. Concretely, this means a divisor $\md$ whose value $\md_v$ on any vertex $v \in G$ satisfies \begin{equation}\label{eq:uniform} 0 \leq \md_v \leq 2 g_v - 2 + \val(v).\end{equation}

If $\md$ is uniform, clearly the class $\delta$ of $\md$ is special. The converse does not hold, and there are special classes which contain no uniform representatives (see Example~\ref{ex:uniform}). 
We show however, that in any special class there are representatives that are almost uniform in the following sense (see Theorem~\ref{thm:main}):

\begin{ThmB} \label{thm: existence special}
    Let $G$ be a graph and $\delta$ a special divisor class. Then $\delta$ contains a representative $\md$ such that for all vertices $v \in V(G)$ we have \begin{equation} -1 \leq \md_v \leq 2g_v - 2 + \val(v).\end{equation}
    Furthermore, $\md_v = -1$ is only possible if $g_v = 0$ and there is no loop adjacent to $v$.
    \end{ThmB}
In particular, if every vertex of $G$ with weight $0$ is adjacent to a loop, then every special divisor class contains a uniform representative. 

While the theorem is purely combinatorial, our original motivation for studying this question is algebro-geometric: Together with results of \cite{Clifford}, Theorem A allows to construct explicit Clifford representatives for a large class of graphs $G$; that is, divisors $\md \in \delta$, such that every line bundle of multidegree $\md$ satisfies the classical Clifford inequality. We explore this and related questions in \cite{BCM}. 


\medskip

\noindent 
{\bf Acknowledgements.} I would like to thank Sam Payne for pointing out Luo's work on $\mathcal E$-reduced divisors, and Matt Baker for comments on an earlier version of the paper. The notion of $V$-reduced divisors was introduced independently by Myrla Barbosa in her PhD thesis \cite{B21}, which was pointed out by Margarida Melo. I'm grateful to both of them for helpful discussions related to this circle of ideas. Finally, I would like to thank an anonymous referee for pointing out the alternative proof of Theorem B using partial orientations explained in Remark~\ref{rmk:partial orientations}.

\section{Divisors and linear equivalence}
\label{sec:divisors}

In this section, we recall some standard notions regarding graphs and divisors on them. 

\subsection{Graphs} Throughout, $G$ will denote a connected, vertex-weighted graph, where we allow for multiple edges and loops. The weight of a vertex $v$ of $G$ is denoted by $g_v$.
The set of edges and vertices of $G$ are denoted by $E(G)$ and $V(G)$, respectively. 

For a subset of vertices $V \subset V(G)$ the \emph{subgraph induced by} $V$ is the subgraph of $G$ with vertex set $V$ and edge set consisting of all the edges of $G$ between vertices in $V$. 

The \emph{genus} of a graph $G$ is defined as \[g = |E(G)| - |V(G)| + 1 + \sum_{v \in V(G)} g_v. \]

Given a vertex $v \in V(G)$, the \emph{valence} $\val(v)$ of $v$ is defined as the number of edges adjacent to $v$ with loop edges based at $v$ counted twice. 

\subsection{Divisors} Let $A = \ZZ$ or $A = \RR$. The group of $A$-\emph{divisors} $\Div_A(G)$ on $G$ is the group of formal linear combinations \[\md = \sum_{v \in V(G)} \md_v \cdot v,\]
with $\md_v \in A$. Clearly, any $\ZZ$-divisor is also an $\RR$-divisor.

Except for Section~\ref{sec:potential}, we will only consider $A = \ZZ$. To simplify notation, we will call $\ZZ$-divisors just divisors and set $\Div(G) = \Div_{\ZZ}(G)$.

The \emph{degree} of an $A$-divisor $\md$ is $\deg(\md) = \sum_{v \in V(G)} \md_v.$ 
An $A$-divisor on $G$ is called \emph{effective}, if its value $\md_v$ is non-negative on each vertex $v$ of $G$. Given a subset of vertices $V \subset V(G)$, we say that an $A$-divisor $\md$ is effective away from $V$ if $\md_v \geq 0$ for all $v \not \in V$. The \emph{support} of an $A$-divisor $\md$ is the set of vertices $v$ on which $\md_v \neq 0$. 

\subsection{Linear equivalence} Next, we describe a well-known 
equivalence relation of divisors on $G$.

We denote by $C^0(G; A)$ the set of $A$-valued functions on vertices of $G$. To any function in $C^0(G; A)$, we can associate an $A$-divisor in $\Div_A(G)$, as follows. 

Given a function $f \in C^0(G; A)$, and a vertex $v \in V(G)$, the $A$-divisor $\dv(f)$ associated to $f$ has value at $v$ equal to the sum of the outgoing slopes of $f$ at $v$, \[\mathrm{div}(f)_v = \sum_{e = vw} f(w) - f(v), \]
where the sum runs over all edges adjacent to $v$. An $A$-divisor of the form $\mathrm{div}(f)$ has always degree $0$ and is called a principal divisor. Two functions $f,g \in C^0(G; A)$ have the same associated $A$-divisor if and only if one function is obtained from the other by adding the same constant at each vertex.

Two divisors $\md, \me \in \Div(G)$ are called \emph{linearly equivalent}, written $\md \sim \me$, if \[\md - \me = \dv(f)\] for some $\ZZ$-valued function $f \in C^0(G; \ZZ)$. 

The set of equivalence classes under this relation is denoted by $\Pic(G)$. Since principal divisors have degree $0$, the classes $\delta \in \Pic(G)$ have a well-defined degree given by the degree of its representatives. We call a class $\delta \in \Pic(G)$ \emph{effective}, if $\delta$ contains an effective representative.   

For a subset of vertices $V \subset V(G)$ we denote by $\mathbb 1_{V}$ the characteristic function of $V$; that is, $\mathbb 1_V$ has value $1$ on all vertices in $V$ and value $0$ on all other vertices. Given a divisor $\md \in \Div(G)$, we say that the linearly equivalent divisor \[\md' = \md + \mathrm{div}(\mathbb 1_{V})\]
is obtained from $\md$ by a \emph{chip-firing move along} $V$. Concretely, $\md'_v = \md_v + k$ where $k$ is the number of edges adjacent between $v$ and vertices in $V$ if $v \not \in V$; and $k$ is minus the number of edges between $v$ and vertices not in $V$ if $v \in V$. Two divisors $\md$ and $\me$ are linearly equivalent if and only if one can be obtained from the other via a series of chip-firing moves.

\section{Dhar decomposition and reduced divisors} \label{sec:Dhar}

In this section, we introduce $V$-reduced divisors for a set of vertices $V \subset V(G)$. They generalize the well-known $v$-reduced divisors, where $v$ is a single vertex. An important tool in their study is the Dhar decomposition, which we again generalize from a single vertex to a set of vertices.

We will assume throughout that the subset of vertices $V \subset V(G)$ is non-empty, without mentioning it each time. 

\subsection{The Dhar decomposition}\label{subsec:Dhar} 
Let $G$ be a graph, $V \subset V(G)$ a set of vertices, and let $\md$ be a divisor effective away from $V$. We define a sequence of subsets of vertices \begin{equation} \label{eq:Dhar decomposition}
    V = V_0 \subset V_1 \subset \ldots \subset V_n,
\end{equation}
iteratively as follows. Let $W_i = V(G) \setminus V_i$ be the complement of $V_i$, and \[\md' = \md + \dv(\mathbb 1_{W_i}),\] the divisor obtained from $\md$ by a chip-firing move along $W_i$. Then $V_{i+1}$ is obtained from $V_i$ by adding those vertices in $V(G) \setminus V_i$ on which $\md'$ is negative. Since there are only finitely many vertices, at some point this process stabilizes and we have $V_{n} = V_{n+1}$. We set \[\W(\md, V) \coloneqq V(G) \setminus V_n\] and call \[V(G) = V_n \sqcup \W(\md, V)\] the \emph{Dhar decomposition of $G$ with respect to $V$ and $\md$.} By construction, the divisor obtained from $\md$ via a chip-firing move along $\W(\md, V)$ is still effective away from $V$.

\medskip

If $V$ contains a single vertex $v$, the above construction gives the usual Dhar decomposition of $G$ with respect to $v$ and $\md$ (see, for example, \cite[\S 3.4]{CLM}). On the other extreme, if $V = V(G)$ we have trivially $\W(\md, V) = \emptyset$  for any divisor $\md$.

\subsection{\texorpdfstring{$V$}{V}-reduced divisors} \label{subsec:reduced}

Reduced divisors are particularly convenient representatives in each divisor class on $G$. Originally, they were defined by fixing a single vertex $v$ \cite[\S 3]{BN}, and we again generalize the notion to arbitrary subsets of vertices.  

\begin{defn}
    Let $V \subset V(G)$. A divisor $\md \in \Div(G)$ is called $V$-reduced, if it is effective away from $V$ and for any subset of vertices $W \subset V(G) \setminus V$ a chip-firing move based at $W$ turns $\md$ negative on at least one of the vertices in $W$. 
\end{defn}

If $V$ contains a single vertex $v$, this definition coincides with the one for $v$-reduced divisors. On the other extreme, if $V = V(G)$ any divisor is $V$-reduced.

\medskip

We next study how the basic properties of $v$-reduced divisors generalize to this setting:

\begin{lma} \label{lma: decomp reduced}
    Let $V \subset V(G)$ be a subset of vertices and $\md \in \Div(G)$ a divisor that is effective away from $V$. Then $\md$ is $V$-reduced if and only if $\W(\md, V) = \emptyset$.
\end{lma}

\begin{proof}
    By construction of $\W(\md, V)$, we have $V \cap \W(\md, V) = \emptyset$ and the divisor obtained from $\md$ by a chip-firing move along $\W(\md, V)$ remains effective away from $V$. Thus if $\md$ is $V$-reduced we need to have $\W(\md, V) = \emptyset$ by the definition of $V$-reduced.

    Conversely, if $\md$ is not $V$-reduced, there is by definition a subset $W \subset V(G) \setminus V$ such that $\md$ remains effective away from $V$ after a chip-firing move along $W$. By construction of $\W(\md, V)$, we need to have $W \subset \W(\md, V)$ and hence in particular $\W(\md, V) \neq \emptyset$, as claimed.
\end{proof}

The following observation is immediate from the definition of $V$-reduced divisors:

\begin{lma} \label{lma: subset vertices}
    Let $V \subset V' \subset V(G)$ be subsets of vertices and $\md$ a divisor on $G$ that is effective away from $V'$. Then if $\md$ is $V$-reduced, it is also $V'$-reduced.
\end{lma}

\begin{lma} \label{lma: existence}
    Any divisor class $\delta$ contains a $V$-reduced representative for any $V \subset V(G)$.
\end{lma}

\begin{proof}
    Let $v \in V$ (recall that we assume $V \neq \emptyset$). 
    By \cite[Proposition 3.1]{BN} any divisor class $\delta$ contains a $v$-reduced representative. Thus the claim follows from Lemma~\ref{lma: subset vertices}. 
\end{proof}

\begin{rmk}
    In \cite[Proposition 3.1]{BN} it is shown that the $v$-reduced representative is unique in its divisor class. It is easy to see, that $V$-reduced divisors need not be unique in their divisor class if $V$ contains more than one vertex.
\end{rmk}

\begin{ex} \label{ex:not v reduced}
    As we saw, every $v$-reduced divisor is $V$-reduced for any set of vertices $V$ that contains the vertex $v$. In the other direction, not every $V$-reduced divisor is $v$-reduced for some $v \in V$. For example, let $G$ be the graph with three vertices $v_1, v_2$ and $v_3$ and two edges between each pair of them, and $\md$ the divisor with value $2$ on each vertex (see Figure~\ref{fig1}). Let $V = \{v_1, v_2\}$ consist of $2$ vertices. Then $\md$ is $V$-reduced, but not $v$-reduced for any vertex $v \in V(G)$. 
\end{ex}

\tikzset{every picture/.style={line width=0.75pt}}
    
    \begin{figure}[ht] 
    \begin{tikzpicture}[x=0.7pt,y=0.7pt,yscale=-0.8,xscale=0.8]
    \import{./}{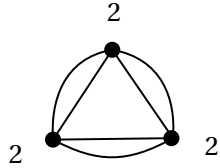}
    \end{tikzpicture}	
    \caption{
     The graph $G$ and divisor $\md$ from Example~\ref{ex:not v reduced}.
    } \label{fig1}
\end{figure}

One property that makes $v$-reduced divisors particularly useful is that a divisor class $\delta$ is effective if and only if its $v$-reduced representative is effective. This enables the use of $v$-reduced divisors to study effectiveness (or more generally, the Baker-Norine rank) of divisor classes. 

This does not generalize to $V$-reduced divisors: a $V$-reduced divisor can fail to be effective, even though its class is effective. An immediate example is given by $V = V(G)$ and $\md$ not effective but with effective class $\delta$. Examples where $V$ is a proper subset of $V(G)$ are also easy to find:  
\begin{ex}
Let $G$ be the graph from Example~\ref{ex:not v reduced}, $\md = (-2,2,2)$, and $V = \{v_1, v_2\}$. Then $\md$ is $V$-reduced, and not effective, but linearly equivalent to an effective divisor. 
\end{ex}

A key observation in establishing this property for $v$-reduced divisors is the following: suppose $\md \in \delta$ fails to be effective only at $v$ and $\delta$ is effective. Then $\W(\md, v) \neq \emptyset$ in the Dhar decomposition with respect to $v$ and $\md$. Since the $v$-reduced divisor in $\delta$ can be obtained by successive chip-firing moves along the sets $\W(\md, v)$, this leads to the conclusion that the $v$-reduced divisor in $\delta$ needs to  be effective if $\delta$ is. 

This observation does generalize to $V$-reduced divisors, and will be central for our further arguments:

\begin{prop} \label{prop:Dhar decomp effective}
    Let $\md$ be a divisor on $G$ that is not effective but whose linear equivalence class $\delta$ is effective. Let $V \subset V(G)$ be the set of vertices, on which $\md$ has negative coefficients. Then \[\W(\md, V) \neq \emptyset.\] 
\end{prop}

\begin{proof}
    By assumption, $\md$ is linearly equivalent to an effective divisor $\md'$. Thus there is an integer valued function $f \in C^0(G; \ZZ)$ such that \[\md' = \md + \mathrm{div}(f).\]
    
    Changing $f$ by adding to its value the same integer at each vertex of $G$ does not change $\mathrm{div}(f)$. Thus we may assume that the minimal value of $f$ on vertices of $G$ is $0$. Since $\md$ itself is not effective, $f$ is not zero on all vertices of $G$.
    
    Write $k = \max \{f(v) \mid v \in V(G)\}$ and \[F_i = \left \{ v \in V(G) \mid f(v) \geq i \right \}, \, i = 1, \ldots, k.\] Then \[f = \sum_{i= l}^{k} \mathbb 1_{F_i},\] 
    where $\mathbb 1_{F_i}$ has value $1$ on vertices in $F_i$ and $0$ otherwise. Since $\mathrm{div}(\mathbb 1_{F_i})$ encodes a chip-firing move along $F_i$, this means that $\md'$ is obtained from $\md$ by successively chip-firing along the $F_i$.  

    In particular, since $F_k \subset F_i$ for all $i \leq k$, the vertices in $F_k$ do not receive any chips while moving from $\md$ to $\md'$ in this way. That is, $\md_v' \leq \md_v$ for all $v \in F_k$.  Since $\md'$ is effective and $V$ consists precisely of the vertices on which $\md$ is not effective, it follows that $F_k \cap V = \emptyset$. On the other hand, and again since the vertices in $F_k$ do not receive any chips and $\md'$ is effective, we have that the divisor obtained from $\md$ via a chip-firing move along $F_k$ remains effective at vertices in $F_k$ and hence effective away from $V$. Thus $\md$ is not $V$-reduced and the claim follows from Lemma~\ref{lma: decomp reduced}. 
\end{proof}

\begin{cor}
    Let $\md$ be a divisor that is not effective and $V \subset V(G)$ the set of vertices where $\md$ has negative coefficients. If $\md$ is $V$-reduced, then the class $\delta$ of $\md$ is not effective.  
\end{cor}

\begin{proof}
    This is immediate from Lemma~\ref{lma: decomp reduced} and Proposition~\ref{prop:Dhar decomp effective}.
\end{proof}

\subsection{An algorithm to find effective representatives}

The previous results allow us to check whether a given divisor is linearly equivalent to an effective divisor, and if so find it. 

\begin{Alg}\label{algorithm}
     Given a divisor $\md$ that is not effective, let $V \subset V(G)$ denote the set of vertices on which $\md$ is not effective. If $V = V(G)$, there is no effective divisor linearly equivalent to $\md$ since in this case $\deg(\md) < 0$. 
     
     Let $V(G) = V_n \cup \W(\md, V)$ be the Dhar decomposition with respect to $V$ and $\md$ introduced in Section~\ref{subsec:Dhar}. If $\W(\md, V) = \emptyset$, $\md$ is not linearly equivalent to an effective divisor by Proposition~\ref{prop:Dhar decomp effective}.  Otherwise, we perform a chip-firing move along $\W(\md, V)$ obtaining a divisor $\md_1 \sim \md$. 
    
    If $\md_1$ is effective, we are done. Otherwise, let $V_1$ denote the set of vertices on which $\md_1$ fails to be effective. Since $\md_1$ was obtained from $\md$ by a chip-firing move along $\W(\md, V)$, and $\md_1$ is effective on vertices of $\W(\md, V)$ by the construction of $\W(\md, V)$, we have $V_1 \subset V$. We repeat the above construction for $\md_1$ and $V_1$ to obtain $\md_2$ and so on. In the end we obtain either $\md_i \sim \md$ with $\md_i$ effective, or $\W(\md_i, V) = \emptyset$ and hence $\md$ is not linearly equivalent to any effective divisor by Proposition~\ref{prop:Dhar decomp effective}.
\end{Alg}
    
\begin{proof}[Proof of termination]
    It suffices to show that eventually either $\W(\md_i, V_i) = \emptyset$ or $V_{i} \neq V_{i + 1}$. In the first case, the algorithm terminates immediately and the second suffices to prove eventual termination since $V_{i + 1} \subset V_{i}$ and there are only finitely many vertices in $V$.
    
    For a vertex $v \in V(G)$ denote by $d(v)$ the length of the shortest path between $v$ and a vertex contained in $V$. Let $l = \max \{d(v) \mid v \in V(G)\}$ and set \[S_k = \{ v \in V(G) \mid d(v) = k \} \] 
    for $k = 0, \ldots, l$. Associated to a multidegree $\md$ we define \[\mu(\md) = \left ( \sum_{v \in S_0} \md_v, \ldots, \sum_{v \in S_l} \md_v \right ).\]
    Then the lexicographic order on the entries in $\mu(\md)$ defines a partial order on the set of degree $d$ divisors $\Div^d(G)$ on $G$. 

    We claim that if $\md_{1}$ is obtained from $\md$ by a chip-firing move along $\W(\md, V)$ as above, then $\md_{1} > \md$ in this partial order. Indeed, let $k_0$ be the length of a path of minimal length between a vertex in $V$ and a vertex in $\W(\md, V)$. That is, $k_0$ is the minimal value for which $\W(\md, V) \cap S_{k_0} \neq \emptyset$. Then a chip-firing move of $\md$ along $\W(\md, V)$ increases the value of $\sum_{v \in S_{k_0 - 1}} (\md)_v$ while not decreasing the value of $\sum_{v \in S_k} (\me_i)_v$ for any $k < k_0 - 1$. Hence $\md_{1} > \md$. 

    Finally, there are only finitely many divisors $\md'$ of fixed degree $d$ that are effective outside $V$ and such that $\sum_{v \in V} \md_v \leq \sum_{v \in V} \md'_v < 0$. By construction, $\md_1$ satisfies these conditions as long as $V = V_1$ and hence the sequence \[\md < \md_1 < \md_2 \ldots\]
    needs to terminate; that is, eventually either $\W(\md_i, V_i) = \emptyset$ or $V_{i} \neq V_{i + 1}$.
\end{proof}

\begin{rmk}
    The above algorithm is a version of \cite[Algorithm 4]{BS13}, which is as far as we know the most common way to check algorithmically, whether a divisor $\md$ is linearly equivalent to an effective divisor. The algorithm in \emph{loc. cit.} calculates the $v$-reduced representative in three steps: 
    \begin{enumerate}
        \item Replace $\md$ by a divisor $\md$ such that $|\md_w| < \val(w)$ for all $w \neq v$.
        \item Calculate $\md' \sim \md$ such that $\md'$ is effective away from $v$.
        \item Calculate the $v$-reduced representative by running the above algorithm for $V = \{v\}$.
    \end{enumerate}

    The difference in Algorithm~\ref{algorithm} is that Step 2 is avoided by directly calculating $V_i$-reduced representatives for successive $V_i$. The first step in \cite[Algorithm 4]{BS13} is used to bound the appearing coefficients of $\md$ and thus reduce computing times. 
\end{rmk}

\section{Potential theory} \label{sec:potential}

    In this section, we compare $V$-reduced divisors to the $\mathcal E$-reduced divisors introduced by Luo \cite{L13} \cite{L18}. This generalizes an observation of Baker and Shokrieh \cite{BS13} that $v$-reduced divisors minimize a certain function on the set of all degree $d$ divisors effective away from $v$.

    \subsection{\texorpdfstring{$\mathcal E$}{E}-reduced divisors}

    In this section, we need to consider divisors with values not only in $\mathbb Z$ but also in $\mathbb R$. 
    Throughout the section, $\RR$-divisors are given by calligraphic upper case letters, whereas $\ZZ$-divisors -- which we will keep calling just divisors -- are given, as before, by underlined lower case letters. Clearly, any divisor is in particular an $\RR$-divisor. 
    
    Recall from Section~\ref{sec:divisors} that an $\RR$-divisor $\mathcal E$ is a formal linear combination \[\mathcal E = \sum_{v \in V(G)} \mathcal E_v v,\]
    with coefficients $\mathcal E_v \in \RR$. 
    
    Recall that $C^0(G; \RR)$ denotes the set of $\RR$-valued functions with domain $V(G)$. For a function $f \in C^0(G; \RR)$ we denote by $U_0(f)$ the set of vertices $v$ on which $f(v) = 0$. 
    
    Given any two $\RR$-divisors $\mathcal E$ and $\mathcal D$ of the same degree, there exists a function $f \in C^0(G; \RR)$ such that $\mathcal E = \mathcal D + \mathrm{div}(f)$; this function is unique up to addition of the same scalar to the value on each vertex of $G$ \cite[Theorem 6]{BF06}. 
    
    Let \[q_{\mathcal E}(\mathcal D)\colon V(G) \to \RR \] be the unique function such that \[\mathcal E = \mathcal D + \mathrm{div}\left(q_{\mathcal E}(\mathcal D)\right)\] and the minimal value of $q_{\mathcal E}(\mathcal D)$ equals $0$. 
    
    \begin{defn}
        Let $\mathcal E$ be an $\RR$-divisor of degree $d$ supported on a set of vertices $V \subset V(G)$. A divisor $\md$ of the same degree $d$ is called $\mathcal E$\emph{-reduced} if it is effective away from $V$ and the value of \[\sum_{v \in V(G)} q_{\mathcal E}(\md)(v) \]
        is minimal among all divisors $\md'$ that are linearly equivalent to $\md$ and effective away from $V$. 
    \end{defn}

    \begin{rmk}
        The notion of $\mathcal E$-reduced divisors was introduced in \cite{L13}, where it is shown that every linear system contains a unique $\mathcal E$-reduced representative. This follows from \cite[Theorem 5.3]{L13} and the observation that the locus \[\left\{ \md' \mid \md' \sim \md \text { and } \md' \text{ is effective away from } V \right \}\] is, in the language of \emph{loc. cit.}, tropically convex and compact.

        As observed already in \cite[Lemma 4.8]{BS13}, $v$-reduced divisors are the special case of $\mathcal E$-reduced divisors for $V = \{v\}$ and $\mathcal E = d \cdot v$.  
    \end{rmk}
    
    \subsection{Preliminary observations}

    Before we can clarify the relation between $\mathcal E$- and $V$-reduced divisors, we need to collect some auxiliary results. Recall that for $f \in C^0(G; \RR)$ we denote by $U_0(f)$ the set of vertices on which $f$ has value $0$.
    
    \begin{lma} \label{lma:chip firing dynamics}
        Let $\mathcal E$ be an $\RR$-divisor of degree $d$ and $\md \sim \md'$ two linearly equivalent divisors of degree $d$. Let $f \in C^0(G; \ZZ)$ be the integer-valued function on $V(G)$ such that \[\md' = \md + \dv(f)\] and the minimal value of $f$ is $0$.
        Assume that $U_0(f) \cap U_0\left(q_{\mathcal E}(\md')\right) \neq \emptyset$. 
        Then \[q_{\mathcal E}(\md) = q_{\mathcal E}(\md') + f.\]
    \end{lma}

    \begin{proof}
        We have by definition and the assumptions of the lemma: \[\mathcal E = \dv \left(q_{\mathcal E}(\md')\right) + \md' = \dv \left(q_{\mathcal E}(\md')\right) + \dv\left(f \right) + \md = \dv \left(q_{\mathcal E}(\md') + f\right) + \md.\]
        Since $U_0(f) \cap U_0(q_{\mathcal E}(\md')) \neq \emptyset$, the minimal value of $q_{\mathcal E}(\md') + f$ is $0$. Thus by definition \[q_{\mathcal E}(\md) = q_{\mathcal E}(\md') + f.\]
    \end{proof}

    \begin{cor} \label{cor:chip firing dynamics}
        Let $\mathcal E$ be an $\RR$-divisor of degree $d$ and $\md \sim \md'$ be two linearly equivalent divisors of degree $d$ such that $\md'$ is obtained from $\md$ by a chip-firing move along a set of vertices $V'$. Suppose $U_0(q_{\mathcal E}(\md')) \not \subset V'$. Then \[\sum_{v \in V(G)} q_{\mathcal E}(\md)(v) = |V'| 
        + \sum_{v \in V(G)} q_{\mathcal E}(\md')(v).\]
    \end{cor}

    \begin{proof}
        Denote by $\mathbb 1_{V'}$ the characteristic function of $V'$ (i.e. with value $1$ on each vertex in $V'$ and value $0$ on all other vertices). By definition, we then have \[\md' =  \md + \mathrm{div}\left(\mathbb 1_{V'} \right).\]
        Since $U_0(q_{\mathcal E}(\md')) \not \subset V'$, we have \[U_0\left(q_{\mathcal E}(\md')\right) \cap U_0\left(\mathbb 1_{V'}\right) = U_0\left(q_{\mathcal E}(\md')\right) \cap \left (V(G) \setminus V' \right) \neq \emptyset.\] Setting $f = \mathbb 1_{V'}$ in Lemma~\ref{lma:chip firing dynamics} gives the claim.
    \end{proof}

\begin{lma}\label{lma: interpolating function}
    Let $G$ be a graph, $v \in V$ a fixed vertex and $V \subset V(G)\setminus \{v\}$ a set of vertices not containing $v$. Then for any $k \in \mathbb R$ there is an $\RR$-valued function $f \in C^0(G; \RR)$ on $V(G)$ such that
    \begin{enumerate}
        \item The value of $f$ on vertices in $V$ is $k$,
        \item the value of $f$ on $v$ is $0$, 
        \item and for all vertices $u \neq v$ and $u \not \in V$, the sum of the outgoing slopes of $f$ at $u$ is $0$, i.e., \[\sum_{e = uw} f(w) - f(u) = 0, \] where the sum runs over all edges adjacent to $u$.
    \end{enumerate}
\end{lma}

\begin{proof}
    For any $v_i \in V$ let $q_{v}(v_i) \in C(G; \RR)$ denote the unique function such that $v = v_i + \dv\left(q_v(v_i)\right)$ and such that the minimal value of $q_{v}(v_i)$ is $0$, which is achieved at $v$. By construction, for any vertex $u$ different from $v$ and $v_i$, the sum of the outgoing slopes of $q_{v}(v_i)$ at $u$ is zero. Thus any linear combination \begin{equation}\label{eq:condition} f = \sum_i f_i q_{v}(v_i), \end{equation}
    $f_i \in \RR$ satisfies the second and the third claim. 

    To show that we can ensure also the first claim, we first observe that the $q_{v}(v_i)$ are linearly independent (in $C^0(G; \RR) \simeq \RR^{|V(G)|}$). Indeed, let $L$ be the linear span of the $q_{v}(v_i)$. Then $\dv(\cdot )$ induces a surjective linear map \[L \to \RR^{|V|}, \, f \mapsto \left(\dv(f)_{v_1}, \ldots, \dv(f)_{v_{|V|}}\right).\] Since there are $|V|$ functions $q_{v}(v_i)$, this shows that they need to be linearly independent. 

    Now let $\RR^{|V| + 1}$ be the vector space spanned by 
    the constant function with vaue $1$, $\mathbb 1_{V(G)}$, and the $q_{v}(v_i)$. Let $\mathrm{ev}_i \colon \RR^{|V| + 1} \to \RR$ be the linear map given by evaluation at $v_i$ and let $K$ be the intersection \[K = \bigcap_i \ker(\mathrm{ev}_i).\] Thus $K$ is a linear subspace of dimension at least $1$. 
    
    As we observed, any linear combination of the $q_{v}(v_i)$ satisfies the second and third claim. It follows that if such a linear combination vanishes on all vertices in $V$, then it vanishes on all vertices in $V(G)$. Hence we have $K \cap \{x_1 = 0\} = \{0\}$, where  $x_1$ is the coordinate of $\RR^{|V| + 1}$ that corresponds to the constant function $\mathbb 1_{V(G)}$. But then $\dim(K) = 1$ and also $K \cap \{x_1 = -k\}$ is a point, and in particular not empty, since $\{x_1 = -k\}$ is a translate of $\{x_1 = 0\}$. By construction, forgetting the $x_1$-coordinate of any point in $K \cap \{x_1 = -k\}$ gives a linear combination $f$ of the $q_{v}(v_i)$ that satisfies the first claim.
\end{proof}

\begin{rmk} \label{rmk:slopes}
    The third claim in Lemma~\ref{lma: interpolating function} in fact does not just hold for vertices $u \neq v$ and $u \not \in V$, but holds also for vertices $u \in V$ such that there is no simple path from $v$ to $u$ that does not contain any other vertex of $V$. Indeed, the first and third claim imply that $f$ needs to have constant value $k$ on any connected component of $G - V$ that does not contain $v$. Hence $f$ has value $k$ on all vertices adjacent to a vertex $u$ as above. Since $f$ has value $k$ also on $u$, this implies the third claim.   
\end{rmk}

\begin{lma} \label{lma:zero set}
     Let $V \subset V(G)$ be a set of vertices and $\mathcal D$ an $\RR$-divisor of degree $d$, effective away from $V$. Let $\mathcal E$ be an $\RR$-divisor of degree $d$ supported on $V$. Then 
     \begin{enumerate}
         \item $U_0(q_{\mathcal E}(\mathcal D)) \cap V \neq \emptyset$
         \item If $v \in U_0(q_{\mathcal E}(\mathcal D))$ and $v \not \in V$, then $U_0(q_{\mathcal E}(\mathcal D))$ contains the entire connected component $G'$ of $G - V$ that contains $v$, as well as any vertex in $V$ adjacent to a vertex in $G'$.
     \end{enumerate}
\end{lma}

\begin{proof}
    Recall that $q_{\mathcal E}(\mathcal D)$ is normalized so that the vertices at which it is $0$ are precisely those at which it achieves its minimum value and that we have by definition \[\mathcal E = \mathcal D + \dv \left(q_{\mathcal E}(\mathcal D)\right) .\] Let $v \in V(G) \setminus V$. If $v$ is a local minimum of $q_{\mathcal E}(\mathcal D)$, this means that $\mathcal D_v \leq \mathcal E_v$. Since $\mathcal D$ is effective outside $V$, we have $0 \leq \mathcal D_v$; since $\mathcal E$ is supported on $V$, we have $\mathcal E_v = 0$. It follows that \[\mathcal E_v = \mathcal D_v = 0.\]
    Thus the sum of the outgoing slopes of $q_{\mathcal E}(\mathcal D)$ at $v$ is $0$. Since $v$ is assumed to be a local minimum of $q_{\mathcal E}(\mathcal D)$, this shows that the value of $q_{\mathcal E}(\mathcal D)$ needs to be the same on $v$ as on all vertices adjacent to $v$. Repeating this procedure for vertices in $V(G) \setminus V$ that are adjacent to $v$ shows the second claim. For the first claim, observe that since $G$ is connected, eventually a vertex contained in $V$ needs to be adjacent to the vertex $v$.
\end{proof}

\begin{lma} \label{lma:R divisor}
    Let $V \subset V(G)$ be a set of vertices and $\mathcal D$ an $\RR$-divisor of degree $d$, effective away from $V$. Then there exists an $\RR$-divisor $\mathcal E$ of degree $d$, supported on $V$ and such that $q_{\mathcal E}(\mathcal D)$ has value $0$ on all vertices in $V$. 
\end{lma}

\begin{proof}
    We show the claim by induction on the number of vertices in $V - U_0\left(q_{\mathcal E}(\mathcal D)\right)$. 
    The base case is $V - U_0\left(q_{\mathcal E}(\mathcal D)\right) = \emptyset$, in which case the claim holds by definition. 

    For the inductive step, let $v \in V - U_0\left(q_{\mathcal E}(\mathcal D)\right)$ be a vertex in $V$ on which $q_{\mathcal E}(\mathcal D)$ is not zero and hence positive. By Lemma~\ref{lma: interpolating function} there exists an $\RR$-valued function $f_k \in C(G; \RR)$ that has value $k$ on $U_0\left(q_{\mathcal E}(\mathcal D)\right)$, value $0$ on $v$ and such that the sum of the outgoing slopes of $f_k$ is $0$ at any vertex $u$ different from $v$ and not contained in $U_0\left(q_{\mathcal E}(\mathcal D)\right)$. Set \[\mathcal E_k = \mathcal E + \dv(f_k). \]
    We claim that $\mathcal E_k$ is still supported on $V$. Indeed, by construction of $f_k$ and Remark~\ref{rmk:slopes}, the only vertices $u$ such that $(\mathcal{E}_k)_u \neq \mathcal{E}_u$ are $u = v$ and vertices in $U_0\left(q_{\mathcal E}(\mathcal D)\right)$ that are connected to $v$ by a simple path that does not contain any other vertex in $U_0\left(q_{\mathcal E}(\mathcal D)\right)$. But $v \in V$ by assumption and any vertex satisfying the latter property is contained in $V$ by Lemma~\ref{lma:zero set}. 

    Since $U_0\left(q_{\mathcal E}(\mathcal D)\right)$ is the set of vertices on which $q_{\mathcal E}(\mathcal D)$ has minimal value (equal to $0$), $q_{\mathcal E} + f_k$ still has minimal value (equal to $k$) on $U_0\left(q_{\mathcal E}(\mathcal D)\right)$ as long as $k$ is positive but small enough.
    It follows that \[q_{\mathcal E_k}(\mathcal D) = q_{\mathcal E}(\mathcal D) + f_k - k.\] 
    Increasing $k$ thus at first preserves $U_0\left(q_{\mathcal E_k}(\mathcal D)\right)$ until eventually the minimum is achieved at additional vertices (the values of $q_{\mathcal E_k}(\mathcal D)$ vary continuously with $k$). By Lemma~\ref{lma:zero set} at least one of these new vertices needs to be contained in $V$ and this gives the inductive step.  
\end{proof}

\subsection{The relation between \texorpdfstring{$\mathcal E$}{E}- and \texorpdfstring{$V$}{V}-reduced divisors}

We are now ready to state and prove the main result of this section: 

\begin{thm} \label{thm: equivalence reducedness}
    A divisor $\md$ of degree $d$ is $V$-reduced if and only if there exists an $\RR$-divisor $\mathcal E$ of degree $d$ which is supported on $V$ and such that $\md$ is $\mathcal E$-reduced. 
\end{thm}

\begin{proof}
    Assume that $\md$ is $\mathcal E$-reduced for some $\RR$-divisor $\mathcal E$ of degree $d$ which is supported on $V$. If $\md$ is not $V$-reduced, there exists a subset $V'$ of vertices, such that $V' \cap V = \emptyset$ and the divisor $\md'$ obtained from $\md$ by a chip-firing move along $V'$ is still effective outside of $V$. By Lemma~\ref{lma:zero set}, $U_0(q_{\mathcal E}(\md')) \not \subset V'$, that is, $V'$ does not contain all vertices on which $q_{\mathcal E}(\md')$ achieves its minimum. By Corollary~\ref{cor:chip firing dynamics}, this implies \[\sum_{v \in V(G)} q_{\mathcal E}(\md')(v) < \sum_{v \in V(G)} q_{\mathcal E}(\md)(v),\] which contradicts the assumption that $\md$ is $\mathcal E$ reduced.

    \medskip
    
    Now assume conversely that $\md$ is $V$-reduced. In particular, $\md$ is effective away from $V$. Let $\mathcal E$ be the $\RR$-divisor constructed in Lemma~\ref{lma:R divisor}. That is, $\mathcal E$ is supported on $V$, has degree $d$ and $q_{\mathcal E}(\md)$ has value $0$ on all vertices in $V$. We claim that $\md$ is $\mathcal E$-reduced.   

    Indeed, suppose there is another divisor $\md' \sim \md$ such that $\md'$ is effective away from $V$ and \begin{equation} \label{eq:E reduced} \sum_{v \in V(G)} q_{\mathcal E}(\md')(v) < \sum_{v \in V(G)} q_{\mathcal E}(\md)(v).\end{equation}
    Set \[\md = \dv(f) + \md',\] where the minimal value of $f$ is $0$.
    
    If $U_0(f) \cap U_0(q_{\mathcal E}(\md)) \neq \emptyset$ we would get \[q_{\mathcal E}(\md')(v) =  q_{\mathcal E}(\md)(v) + f(v) \geq q_{\mathcal E}(\md)(v)\] for all $v \in V(G)$ by Lemma~\ref{lma:chip firing dynamics}, and hence a contradiction to \eqref{eq:E reduced}.
    
    So suppose $U_0(f) \cap U_0(q_{\mathcal E}(\md)) = \emptyset$. By the choice of $\mathcal E$ we have $V \subset U_0(q_{\mathcal E}(\md))$, and hence $U_0(f) \cap V = \emptyset$. 
    If $k$ is the maximal value of $f$, we have \[\md' = \md - \dv(f) = \md + \dv(-f) = \md + \dv\left(-f + k \cdot \mathbb 1_{V(G)}\right),\]
    where $-f + k \cdot \mathbb 1_{V(G)}$ is the function obtained from $-f$ by adding $k$ to its value at each vertex. Thus $-f + k \cdot \mathbb 1_{V(G)}$ has minimal value $0$ and maximal value $k$, the latter achieved exactly on vertices in $U_0(f)$.
    
    As in the proof of Proposition~\ref{prop:Dhar decomp effective}, write \[F_i = \left \{ v \in V(G) \mid -f(v) + k \cdot \mathbb 1_{V(G)} \geq i\right\}, \, i = 1, \ldots, k,\] such that \[-f + k \cdot \mathbb 1_{V(G)} = \sum_{i= l}^{k} \mathbb 1_{F_i}.\] 
    Since $\mathrm{div}(\mathbb 1_{F_i})$ encodes a chip-firing move along $F_i$, this means that $\md'$ is obtained from $\md$ by successively chip-firing along the $F_i$. In particular, $\md'_v \leq \md_v$ for all vertices $v \in F_k$ since $F_k \subset F_i$ for all $i \leq k$. Thus the divisor obtained from $\md$ by a chip-firing move along $F_k$ is effective away from $V$, since so is $\md'$. But we also have \[F_k \cap V = U_0(f) \cap V = \emptyset,\] which contradicts the assumption that $\md$ is $V$-reduced. 
\end{proof}

    To conclude this section we note that it is in general not possible to choose the $\RR$-divisor $\mathcal E$ in Theorem~\ref{thm: equivalence reducedness} with integer coefficients:

    \begin{ex} \label{ex:integer E}
        Consider the graph $G$ with $3$ vertices $v_1, v_2, v_3$ and four edges, two of them adjacent to $v_2$ and $v_1$, and two to $v_2$ and $v_3$ (see Figure~\ref{fig2}). Let $V = \{v_1, v_3\}$ and $\md = (1,3,2)$. Then $\md$ is $V$-reduced, since the only vertex not in $V$ is $v_2$ and a chip-firing move along $v_2$ turns $\md_{v_2}$ negative. 
        
        The $\RR$-divisor $\mathcal E$ constructed in the proof of Theorem~\ref{thm: equivalence reducedness} is $(2.5, 0, 3.5)$. Then $q_{\mathcal E}(\md)$ is $0$ on both $v_1$ and $v_3$ and has outgoing slope at $v_1$ and $v_3$ equal to $0.75$ along each edge. 
        
        On the other hand, if $\mathcal E$ has integer values we claim $\md$ cannot be $\mathcal E$-reduced. Indeed, if $\mathcal E = (3,0,3)$, then a chip-firing move along $\{v_2, v_3\}$ 
        produces the divisor $\md' = (3,1,2)$ and one checks that $\sum_{v \in V(G)} q_{\mathcal E}(\md)(v) > \sum_{v \in V(G)} q_{\mathcal E}(\md')(v)$. Thus $\md$ is not $\mathcal E$-reduced. If $\mathcal E \neq (3,0,3)$ but has integer values, degree $6$ and is supported on $V$, $\mathcal E$ has value at least $4$ on one of the vertices $v_1$ or $v_3$. A chip-firing move towards this vertex again decreases $\sum_{v \in V(G)} q_{\mathcal E}(\md)(v)$  showing that $\md$ is not $\mathcal E$-reduced also in this case.
    \end{ex}

    \tikzset{every picture/.style={line width=0.75pt}}
    
    \begin{figure}[ht] 
    \begin{tikzpicture}[x=0.7pt,y=0.7pt,yscale=-0.8,xscale=0.8]
    \import{./}{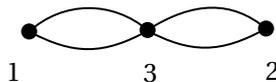}
    \end{tikzpicture}	
    \caption{
     The graph $G$ and divisor $\md$ from Example~\ref{ex:integer E}.
    } \label{fig2}
\end{figure}

\section{Uniform representatives}
\label{sec:uniform}

In this section, we recall the definition of uniform divisors and special divisor classes. Then we study under which conditions on $G$ any special divisor class contains a uniform representative. 

\subsection{Uniform divisors and special divisor classes}

Recall that $g_v$ denotes the weight of a vertex $v$ and $\val(v)$ denotes the valence of a vertex $v \in V(G)$, that is, the number of edges adjacent to $v$ with loops counted twice.
\begin{defn}
    The \emph{canonical divisor} $\mk$ on $G$ is the divisor with value on a vertex $v$ given by \[2g_v-2 + \val(v).\]
\end{defn}
The degree of the canonical divisor is $2g - 2$, where $g$ denotes the genus of the graph. Indeed, we have 
\begin{eqnarray*}
    \deg(\mk) &=& \sum_{v \in V(G)} \left(2 g_v - 2 + \val(v)\right) = 2 |E(G)| - 2 |V(G)| + 2 \sum_{v\in V(G)} g_v \\ &=& 2(|E(G)| - |V(G)| + 1 + \sum_{v \in V(G)} g_v) - 2 = 2g - 2 
\end{eqnarray*}

Given a divisor $\md \in \Div(G)$, its \emph{residual divisor} is  \[\md^* \coloneqq \mk - \md.\]
Similarly, if $\delta$ is the class of $\md$ in $\Pic(G)$, then the \emph{residual class} $\delta^*$ of $\delta$ is the class of the residual $\md^*$ of $\md$. It is easy to see, that this class does not depend on the choice of the representative $\md$ in $\delta$. 

\medskip

Recall that a divisor class $\delta$ is effective if it contains an effective representative.

\begin{defn}
    Let $G$ be a graph.
    \begin{enumerate}
        \item A divisor $\md \in \Div(G)$ is \emph{uniform}, if both $\md$ and its residual divisor $\md^*$ are effective.
        \item A divisor class $\delta \in \Pic(G)$ is \emph{special}, if both $\delta$ and its residual class $\delta^*$ are effective.
    \end{enumerate}
\end{defn}

Explicitly, a divisor $\md$ is uniform if for any vertex $v \in V(G)$, \[0 \leq \md_v \leq 2g_v - 2 + \val(v).\] 
We have $0 > 2g_v - 2 + \val(v)$ if and only if $g_v = 0$ and $\val(v) = 1$ (recall that we assume $G$ to be connected). Thus there exist uniform multidegrees on a graph $G$ if and only if $G$ does not contain any vertex $v$ with $g_v = 0$ and $\val(v) = 1$. Such graphs are called \emph{semistable}.

A divisor $\md$ with negative degree $d$ cannot be effective. Since the degree of the residual $\md^*$ is $2g - 2 - d$, it follows that the degree of a uniform divisor satisfies $0 \leq d \leq 2g - 2$. The same holds for the degree of a special divisor class.

Clearly, the class $\delta$ of a uniform divisor $\md$ is special.
The following example shows that the converse need not hold: 

\begin{ex} \label{ex:uniform}
    \emph{A special divisor class not containing a uniform representative.}
    Consider the graph $G$ with three vertices of weight $0$ and two edges between each pair of vertices (see Figure~\ref{fig1}). The valence of each vertex is $4$ and thus the canonical divisor on $G$ is $(2,2,2)$. 
    
    Then the divisor $\md = (0,0,3)$ is not uniform, since \[\md^* = \mk - \md = (2,2,-1)\] is not effective. 
    Furthermore, $\md$ is the only effective representative in its class $\delta$, and thus $\delta$ does not contain a uniform representative.
    
    But $\delta$ is special since \[\md^* = (2,2,-1) \sim (0,0,3)\] via a chip-firing move along the two vertices on which $\md^*$ has value $2$. 
\end{ex}

\begin{rmk}
    Our interest in uniform representatives comes from algebraic geometry. More precisely, results of \cite{Clifford} establish a connection between uniform representatives and line bundles on nodal curves that satisfy the classical Clifford inequality. See \cite{BCM} for a detailed discussion, and, for example, \cite{Caporasocompactification, Baker, CLM, Len17, CC, C24} for more general context on line bundles on nodal curves. 
\end{rmk}

\subsection{Uniform representatives in special divisor classes}

    We are ready to give the main result of this section, Theorem~\ref{thm:main}. Theorem A in the introduction is the residual formulation of Theorem~\ref{thm:main}, obtained by applying Theorem~\ref{thm:main} to the residual of a special class.  

    \begin{defn}
        For a vertex $v$ denote by $\vall(v)$ the number of non-loop edges adjacent to $v$.
    \end{defn}

    That is, \[\val(v) = \vall(v) + 2l, \]
    where $l$ denotes the number of loops adjacent to $v$.
    
    \begin{thm} \label{thm:main}
         Let $G$ be a graph and $\delta$ a special divisor class. Then $\delta$ contains a representative $\md$ such that for all vertices $v \in V(G)$ we have \begin{equation}\label{eqn:quasi uniform} 0 \leq \md_v \leq \max \left\{\vall(v) - 1, 2g_v - 2 + \val(v) \right\}.\end{equation} 
    \end{thm}
    
    \begin{proof}
    Let $\delta^*$ denote the residual class of $\delta$. By assumption, $\delta^*$ is effective and we may choose an effective representative $\me^*$ in $\delta^*$. By definition, this means that $\me \in \delta$ satisfies \begin{equation} \label{eq:upper bound} \me_v \leq 2 g_v - 2 + \val(v)\end{equation}
    for all vertices $v \in V(G)$.
    Thus if $\me$ is effective, we are done. Otherwise, denote by $V \subset V(G)$ the subset of vertices at which $\me$ fails to be effective. Since $\delta$ is assumed to be effective, we need to have $d \geq 0$ and hence $V \neq V(G)$. 

    We now run Algorithm~\ref{algorithm} on $\me$. Recall that in the algorithm an effective representative of $\delta$ is constructed from $\me$ by repeated chip-firing moves along Dhar sets $\W(\me, V)$. By construction of $\W(\me, V)$, the divisor $\me'$ obtained from $\me$ by a chip-firing move along $\W(\me, V)$ remains effective away from $V$. Since we assume that $\delta$ is effective and already showed that Algorithm~\ref{algorithm} terminates, it suffices to show that a chip-firing move along $\W(\me, V)$ preserves the upper bound in \eqref{eqn:quasi uniform} (though not necessarily the upper bound for being uniform, \eqref{eq:upper bound}).

    Any vertex $v$ not adjacent to a vertex in $\W(\me, V)$ is not affected by the chip-firing move and hence $\me_v = \me'_v$. Since $\me$ satisfies the upper bound on $v$, so does $\me'$ in this case. Similarly, if $v \in \W(\me, V)$, then $\me_v' \leq \me_v$ and the bound is still satisfied for $\me'$ on $v$.
    
    So suppose $v$ is adjacent to a vertex in $\W(\me, V)$ but not contained in $\W(\me, V)$. If $v \in V$, we have $\me_v < 0$ by the definition of $V$. Since $\me'$ is obtained from $\me$ by a single chip-firing move, we have in this case \[\me_v' \leq \me_v + \vall(v) \leq \vall(v) - 1,\]
    and hence $\me'$ satisfies the claimed upper bound \eqref{eqn:quasi uniform} on $v$.
    
    Finally, the remaining case is $v \not \in V$, $v \not \in \W(\me, V)$ and $v$ adjacent to $\W(\me, V)$. Denote by $k$ the number of edges between $v$ and vertices in $\W(\me, V)$, and by $k^c = \vall(v) - k$ the number of non-loop edges adjacent to $v$ but not to a vertex in $\W(\me, V)$. Thus \begin{equation} \label{eq:sum}
        \me'_v = \me_v + k.
    \end{equation} Since $v$ is not contained in $\W(\me, V)$ and not contained in $V$, we have by construction of the Dhar decomposition that there is a subset of vertices $W_i \subset V(G)$ containing $v$, such that a chip-firing move along $W_i$ turns $\me_v$ negative. Since $\W(\me, V)$ is contained in $W_i$, this implies in particular that \[\me_v \leq k^c - 1.\] 
    Together with \eqref{eq:sum} this gives \[\me'_v = \me_v + k \leq k^c - 1 + k = \vall(v) - 1,\]
    which shows the claimed upper bound for $\me'$ at $v$. 
\end{proof} 

\begin{rmk}
    The upper bound for $\md_v$ in Theorem~\ref{thm:main}, $\max \left\{\vall(v) - 1, 2g_v - 2 + \val(v) \right\}$, is clearly weaker than the upper bound required for uniform divisors. However, the only case in which $\vall(v) - 1 > 2g_v - 2 + \val(v)$ is if $g_v = 0$ and $\vall(v) = \val(v)$; that is, if $v$ has weight zero and no adjacent loop. In this case, the upper bound in Theorem~\ref{thm:main} is higher by $1$ compared to the upper bound for uniform divisors. In all other cases, the bound in the theorem coincides with that for uniform divisors. Thus if every vertex of $G$ with $g_v = 0$ is adjacent to a loop, then every special divisor class contains a uniform representative. 
\end{rmk}

\begin{rmk} \label{rmk:partial orientations}
    An alternative proof of Theorem~\ref{thm:main} can be given using Spencer Backman's theory of partial graph orientations \cite{B17}. We give a sketch of this construction: One may assume that $g_v = 0$ for all vertices $V(G)$ by working on a weightless graph $G'$ obtained from $G$ by adding at each vertex $g_v$ loops. Then linear equivalence on $G$ and $G'$ coincide. Suppose first that we have $d \leq g - 1$. Since $\md$ is effective, we have by \cite[Theorem 4.10]{B17} that $\md$ is given, up to linear equivalence, by a sourceless partial orientation. This means that $0 \leq \md_v \leq \val(v) - 1 + g_v$ with equality in the upper bound only possible when there are no loops based at $v$ and $g_v = 0$. Similarly, for  $d \geq g - 1$, one can use \cite[Algorithm 4.3]{B17} to show that the residual $\md^*$ can be given, up to linear equivalence, by a sinkless partial orientation. Furthermore, this partial orientation can be arranged to have no sources at vertices that are adjacent to a loop edge. It follows that $-1 \leq \md^*_v \leq \val(v) - 2 + g_v$ where the lower bound can only be achieved at weightless vertices not adjacent to any loop edges. This in turn gives $g_v \leq \md_v \leq 2 g_v - 1 + \val(v)$, where equality in the upper bound again is only possible when there are no loops based at $v$ and $g_v = 0$.
\end{rmk}

\providecommand{\bysame}{\leavevmode\hbox to3em{\hrulefill}\thinspace}
\providecommand{\MR}{\relax\ifhmode\unskip\space\fi MR }
\providecommand{\MRhref}[2]{%
  \href{http://www.ams.org/mathscinet-getitem?mr=#1}{#2}
}
\providecommand{\href}[2]{#2}

\end{document}